%% file: main.tex
\documentclass[reqno,10pt]{article}

\usepackage{graphicx}
\usepackage{amsmath}
\usepackage{amssymb}
\usepackage{amsthm}
\usepackage{enumitem}
\usepackage{url}
\usepackage[colorlinks=true,linkcolor=blue,citecolor=blue,urlcolor=blue]{hyperref}

\newtheorem{theorem}{Theorem}
\newtheorem{lemma}[theorem]{Lemma}
\newtheorem{proposition}[theorem]{Proposition}
\newtheorem{corollary}[theorem]{Corollary}

\theoremstyle{definition}
\newtheorem{definition}[theorem]{Definition}
\newtheorem{remark}[theorem]{Remark}
\newtheorem{example}[theorem]{Example}

\numberwithin{equation}{section}
\numberwithin{theorem}{section}

\newenvironment{OMabstract}{\noindent\textbf{Abstract.} }{\medskip}
\newenvironment{OMsubjclass}{\noindent\textbf{Mathematics Subject Classification (2020):} }{\medskip}
\newenvironment{OMkeywords}{\noindent\textbf{Keywords:} }{\medskip}
\newcommand{\ldc}{\operatorname{ldc}}
\newcommand{\diam}{\operatorname{diam}}
\newcommand{\ecc}{\varepsilon}
\newcommand{\med}{\operatorname{med}}
\newcommand{\Ball}[2]{N^{#2}[#1]}
\newcommand{\Tkh}{T^{k,h}}
\newcommand{\Ntwo}{M}

\begin{document}

\author{Fei-Huang Chang, Ma-Lian Chia, David Kuo, and Guan-Ting Lai}
\title{List-distance consistent vertices in trees are confined to a path}
\date{}
\maketitle

\hrule\smallskip
\begin{center}
Submission for publication in \textsc{Opuscula Mathematica}
\end{center}
\smallskip\hrule\bigskip

\input{abstract}
\input{introduction}
\input{preliminaries}
\input{path-structure}
\input{kary}
\input{spiders}
\input{conclusion}
\section*{Acknowledgements}

The research of D. Kuo was supported in part by the National Science and
Technology Council of Taiwan under grant NSTC 115-2115-M-259-002, and
that of F.-H. Chang under grant NSTC 115-2115-M-259-003.
\input{ai-declaration}

\input{bibliography}
\bigskip
\noindent Fei-Huang Chang, David Kuo, and Guan-Ting Lai\\
Department of Applied Mathematics, National Dong Hwa University,\\
Ma-Lian Chia (corresponding author) \\
Academy of Preparatory Programs for Overseas Compatriot Students,
  National Taiwan Normal University,
New Taipei City 244014, Taiwan\\
\texttt{mlchia@ntnu.edu.tw}

\end{document}

%% file: abstract.tex
\begin{OMabstract}
A \emph{labeling} of a connected graph $G$ on $n$ vertices is a bijection
$c:V(G)\to\{1,\dots,n\}$; writing $c(u,v)=|c(u)-c(v)|$, a vertex $u$ is
\emph{list-distance consistent} if $d(u,v)<d(u,w)$ implies $c(u,v)\le c(u,w)$
for all $v,w$. The maximum number of such vertices over all labelings is the
\emph{list-distance consistency} $\ldc(G)$, introduced by Casselgren and
Henricsson. We prove that in a tree, the consistent vertices of any labeling lie on a single path, along which the labels form a block of consecutive integers in increasing order (with respect to a suitable orientation of the path), no vertex off the path receiving a label from that block. We deduce that $\ldc$ equals $3$ for every complete $k$-ary tree
except the binary tree of height two, and we determine $\ldc$ for all spiders.
\end{OMabstract}

\begin{OMkeywords}
graph labeling, distance labeling, list graph, tree, spider.
\end{OMkeywords}

\begin{OMsubjclass}
05C78, 05C05, 05C85.
\end{OMsubjclass}

%% file: introduction.tex
\section{Introduction}
\label{sec:intro}

All graphs considered here are finite, simple and connected. For a graph $G$,
write $n=|V(G)|$ and let $d(u,v)$ denote the distance between $u$ and $v$. A
\emph{labeling} of $G$ is a bijection $c:V(G)\to[1,n]$, and we abbreviate
$c(u,v)=|c(u)-c(v)|$.

Lennerstad and Eriksson \cite{LE2018} called $G$ a \emph{list graph} if it
admits a labeling such that
\begin{equation}
 d(u_1,v_1)<d(u_2,v_2)\ \Longrightarrow\
 c(u_1,v_1)\le c(u_2,v_2)
 \label{eq:listcond}
\end{equation}
for all $u_1,v_1,u_2,v_2\in V(G)$. They proved that list graphs are
path-hamiltonian, chordal and claw-free; Casselgren and Henricsson \cite{CH2025}
subsequently characterized them as unions of cliques of two consecutive sizes
arranged along a Hamiltonian path. For other graph-labeling problems, see
\cite{Gallian}; distance-labeling schemes, which address a different question,
are surveyed in \cite{GPR}.

We study the local relaxation of \eqref{eq:listcond} introduced in
\cite{CH2025}.

\begin{definition}
\label{def:ldc}
Let $c$ be a labeling of $G$. A vertex $u\in V(G)$ is
\emph{list-distance consistent}, or simply \emph{consistent}, under $c$ if
\begin{equation}
 d(u,v)<d(u,w)\ \Longrightarrow\ c(u,v)\le c(u,w)
 \qquad (v,w\in V(G)).
 \label{eq:ldc}
\end{equation}
Let $S_c(G)$ be the set of consistent vertices under $c$, and define
\[
 \ldc(G)=\max\{\,|S_c(G)|:c\text{ is a labeling of }G\,\}.
\]
A labeling attaining the maximum is \emph{optimal}.
\end{definition}

Casselgren and Henricsson used an apparently weaker definition involving only
consecutive label distances; Proposition~\ref{prop:equiv} shows that it is
equivalent to Definition~\ref{def:ldc}. Thus $\ldc(G)=n$ precisely when $G$ is
a local list graph in their terminology. They determined $\ldc$ for several
standard families and proved that every value between $1$ and $n$ occurs on
some $n$-vertex graph, whereas $\ldc(G)=1$ asymptotically almost surely for
$G\in\mathcal G(n,1/2)$. This led them to ask which graph families are well
described by the model.

Our main result gives a strong answer for trees. For any labeling of a tree,
not merely an optimal one, Theorem~\ref{thm:path} confines all consistent
vertices to a path $P$. The labels on $P$ are consecutive and increasing, and
no vertex outside $P$ receives a label from that interval. Two consequences
make this structure effective: Corollary~\ref{cor:ballgrowth} bounds the growth
of balls around the extremal consistent vertices, while
Corollary~\ref{cor:branch} forces every branch to occur within distance one of
an end of $P$. We apply these results to two contrasting families. Complete
$k$-ary trees branch throughout; their list-distance consistency is $3$, with
the single exception $T^{2,2}$, for which it is $5$. Spiders branch only once;
Theorem~\ref{thm:spidersummary} determines their list-distance consistency
completely, including a finite description of the exceptional parameter
ranges in which the regular value increases by one.

Section~\ref{sec:prelim} develops the ball criterion used throughout.
Section~\ref{sec:path} proves the path structure theorem and its two
consequences. Sections~\ref{sec:kary} and \ref{sec:spiders} treat complete
$k$-ary trees and spiders, respectively, and Section~\ref{sec:conclusion}
records directions for further work. Part of this paper is based on the
master's thesis \cite{Lai2025} of Guan-Ting Lai.

%% file: preliminaries.tex
\section{Preliminaries}
\label{sec:prelim}

For $u\in V(G)$ and an integer $k\ge0$ we write
\[
  \Ball{u}{k}=\{\,x\in V(G) : d(u,x)\le k\,\}
\]
for the closed ball of radius $k$ about $u$. For $A\subseteq V(G)$, write
$c(A)=\{c(x):x\in A\}$. For vertices $u,v$ of a tree $T$, let $P_{uv}$ denote
the unique $u$--$v$ path. Finally, $[i,j]=\{i,i+1,\dots,j\}$.

\subsection{Two equivalent definitions}

In \cite{CH2025} a vertex $u$ is declared consistent when
\begin{equation}
  c(u,w)=c(u,v)+1 \ \Longrightarrow\ d(u,v)\le d(u,w)
  \qquad\text{for all } v,w\in V(G).
  \label{eq:ch}
\end{equation}
Condition \eqref{eq:ch} is formally weaker than \eqref{eq:ldc}, since it
constrains only pairs whose label distances are consecutive.  The two are in
fact equivalent, and we record this once and for all so that results from
\cite{CH2025} may be quoted without further comment.

\begin{proposition}
\label{prop:equiv}
Let $c$ be a labeling of a connected graph $G$ with $n\ge 2$ vertices and let
$u\in V(G)$.  Then $u$ satisfies \eqref{eq:ldc} if and only if $u$ satisfies
\eqref{eq:ch}.
\end{proposition}

\begin{proof}
If \eqref{eq:ldc} holds and $c(u,w)=c(u,v)+1$, the inequality
$d(u,w)<d(u,v)$ would contradict \eqref{eq:ldc}; hence \eqref{eq:ch} holds.

Conversely, assume \eqref{eq:ch}, put $a=c(u)$ and
$M=\max\{a-1,n-a\}$. Since $c$ is a bijection onto $[1,n]$,
\[
  \{\,c(u,x) : x\in V(G)\setminus\{u\}\,\}=[1,M].
\]
If $d(u,v)<d(u,w)$ but $c(u,w)<c(u,v)$, put $t=c(u,w)$ and
$t'=c(u,v)$. Then $1\le t<t'\le M$, and we may choose vertices
$x_t,x_{t+1},\dots,x_{t'}$ with $c(u,x_j)=j$ for every $j$, and we may take
$x_t=w$ and $x_{t'}=v$.  Applying \eqref{eq:ch} to the pair $(x_j,x_{j+1})$ for
$j=t,\dots,t'-1$ gives $d(u,x_j)\le d(u,x_{j+1})$, whence
\[
  d(u,w)=d(u,x_t)\le d(u,x_{t'})=d(u,v),
\]
contradicting $d(u,v)<d(u,w)$.
\end{proof}

\subsection{A characterisation by balls}

The following reformulation of \eqref{eq:ldc} is the tool we use most often.
We adopt the convention $\min\varnothing=+\infty$, so that the condition below is
vacuous once $k\ge\ecc(u)$.

\begin{lemma}
\label{lem:ball}
Let $c$ be a labeling of a connected graph $G$ and let $u\in V(G)$. Then $u$ is
consistent under $c$ if and only if
\begin{equation}
  \max\bigl\{\,c(u,x) : x\in \Ball{u}{k}\,\bigr\}
  \;\le\;
  \min\bigl\{\,c(u,y) : y\in V(G)\setminus \Ball{u}{k}\,\bigr\}
  \label{eq:ballcond}
\end{equation}
for every integer $k\ge0$.
\end{lemma}

\begin{proof}
Suppose $u$ is consistent and let $k\ge0$, $x\in\Ball{u}{k}$ and
$y\notin\Ball{u}{k}$. Then $d(u,x)\le k<d(u,y)$, so $c(u,x)\le c(u,y)$ by
\eqref{eq:ldc}; taking the maximum over $x$ and the minimum over $y$ gives
\eqref{eq:ballcond}.

Conversely, assume \eqref{eq:ballcond} for all $k$ and let $v,w$ satisfy
$d(u,v)<d(u,w)$. Put $k=d(u,v)$. Then $v\in\Ball{u}{k}$ and
$w\notin\Ball{u}{k}$, so $c(u,v)\le c(u,w)$ by \eqref{eq:ballcond}.
\end{proof}

The necessary condition isolated in \cite[Lemma 2.2]{CH2025} is an immediate
consequence.

\begin{corollary}[\cite{CH2025}]
\label{cor:interval}
If $u$ is consistent under $c$, then $c(\Ball{u}{k})$ is a set of consecutive
integers for every $k\ge0$.
\end{corollary}

\begin{proof}
Put $a=c(u)$ and $A=c(\Ball{u}{k})$. If $j\notin A$ lies strictly between
$\min A$ and $\max A$, then $y=c^{-1}(j)\notin\Ball{u}{k}$. For $j>a$, choose
$x\in\Ball{u}{k}$ with $c(x)=\max A$; then $c(u,x)>c(u,y)$, contrary to
\eqref{eq:ballcond}. The case $j<a$ is symmetric.
\end{proof}

We also record the value of $\ldc$ on paths, which is used in
Example~\ref{ex:notsufficient} below and again in Section~\ref{sec:kary}.

\begin{proposition}
\label{prop:path}
For every $n\ge1$ we have $\ldc(P_n)=n$.
\end{proposition}

\begin{proof}
Label the vertices $v_1,\dots,v_n$ in path order. Then
$c(v_i,v_j)=|i-j|=d(v_i,v_j)$, so every vertex is consistent.
\end{proof}

The converse of Corollary~\ref{cor:interval} fails even for a path, so
Lemma~\ref{lem:ball}, rather than intervalhood alone, is needed below.

\begin{example}
\label{ex:notsufficient}
Let $G$ be the path $x_2\,u\,x_1\,y$ and let
\[
  c(x_1)=1,\qquad c(x_2)=2,\qquad c(u)=3,\qquad c(y)=4 .
\]
Then all three sets $c(\Ball{u}{k})$ are intervals, but $u$ is not consistent:
$d(u,x_1)=1<2=d(u,y)$ while $c(u,x_1)=2>1=c(u,y)$. In the language of
Lemma~\ref{lem:ball}, the labels in $\Ball{u}{1}$ are not positioned correctly
about $c(u)$.
\end{example}

Finally we record two facts from \cite{CH2025} that will be used repeatedly.
The first is the construction of \cite[Section 4]{CH2025}: label some leaf $v$
by $1$, its neighbor by $2$, and the remaining vertices with increasing labels
according to their distance to $v$. The second holds because the map
$x\mapsto c(u,x)$ is unchanged when $c$ is reversed.

\begin{proposition}[\cite{CH2025}]
\label{prop:basic}
\begin{enumerate}
\item[\textup{(i)}] $\ldc(T)\ge2$ for every tree $T$.
\item[\textup{(ii)}] If $c$ is a labeling of a connected graph $G$ and $c'$ is
defined by $c'(x)=n+1-c(x)$, then $S_{c'}(G)=S_c(G)$.
\end{enumerate}
\end{proposition}

%% file: path-structure.tex
\section{The path structure theorem}
\label{sec:path}

Throughout this section $T$ denotes a tree. We first prove two lemmas valid
for arbitrary connected graphs.

\begin{lemma}
\label{lem:monotone}
Let $c$ be a labeling of a connected graph $G$ and let $u,v\in S_c(G)$ with
$c(u)<c(v)$. If $P:u=u_1u_2\cdots u_k=v$ is a shortest $u$--$v$ path in $G$,
then
\[
  c(u_1)<c(u_2)<\cdots<c(u_k).
\]
\end{lemma}

\begin{proof}
Suppose not, and let
\[
  m_0=\min\{\,l : 2\le l\le k,\ c(u_{l-1})>c(u_l)\,\}.
\]
By minimality of $m_0$ we have $c(u)=c(u_1)<c(u_2)<\cdots<c(u_{m_0-1})$; in
particular $c(u_{m_0-1})\ge c(u)$, with equality exactly when $m_0=2$.

Since $P$ is a shortest path, $d(u,u_{m_0-1})=m_0-2<m_0-1=d(u,u_{m_0})$, so
consistency of $u$ gives
\[
  c(u_{m_0-1})-c(u)=c(u,u_{m_0-1})\le c(u,u_{m_0})=|c(u)-c(u_{m_0})| .
\]
If $c(u_{m_0})\ge c(u)$, the right-hand side of the last inequality equals
$c(u_{m_0})-c(u)$, which is strictly smaller than $c(u_{m_0-1})-c(u)$ because
$c(u_{m_0})<c(u_{m_0-1})$; this is a contradiction. Hence
\[
  c(u_{m_0})<c(u).
\]
In particular $u_{m_0}\ne v$, since $c(v)>c(u)$; thus $m_0<k$. Now $m_0\ge2$
yields $d(v,u_{m_0})=k-m_0<k-1=d(v,u)$, so consistency of $v$ gives
$c(v,u_{m_0})\le c(v,u)$. But by the inequality just established we have
$c(u_{m_0})<c(u)<c(v)$, so
\[
  c(v,u_{m_0})=c(v)-c(u_{m_0})>c(v)-c(u)=c(v,u),
\]
a contradiction.
\end{proof}

\begin{lemma}
\label{lem:threebranch}
Let $G$ be a connected graph and let $u_1,u_2,u_3$ lie in three distinct
components of $G-w$ for some vertex $w$. Then for every labeling $c$ of $G$ at
least one of $u_1,u_2,u_3$ fails to be consistent under $c$.
\end{lemma}

\begin{proof}
Suppose $u_1,u_2,u_3\in S_c(G)$; without loss of generality
$c(u_1)<c(u_2)<c(u_3)$. As $u_1$ and $u_2$ lie in distinct components of
$G-w$, every $u_1$--$u_2$ path passes through $w$, and $w$ is an interior
vertex of any such path; choosing a shortest one and applying
Lemma~\ref{lem:monotone} gives $c(u_1)<c(w)<c(u_2)$. The same argument applied
to $u_2$ and $u_3$ gives $c(u_2)<c(w)<c(u_3)$. Hence $c(u_2)<c(w)<c(u_2)$, a
contradiction.
\end{proof}

For vertices $x,y,z$ of a tree, let $\med(x,y,z)$ denote their unique median,
the vertex common to $P_{xy}$, $P_{yz}$ and $P_{xz}$; see \cite{Buneman1974}.
Unless the median is one of $x,y,z$, these vertices lie in three distinct
components after it is deleted.

\begin{theorem}
\label{thm:path}
Let $T$ be a tree, let $c$ be a labeling of $T$, and put $S=S_c(T)$. Assume
$|S|\ge2$, and let $a,b\in S$ be the vertices with
\[
  c(a)=\min\{c(x):x\in S\},\qquad c(b)=\max\{c(x):x\in S\}.
\]
Write $P_{ab}=p_1p_2\cdots p_m$ with $p_1=a$ and $p_m=b$. Then
\begin{enumerate}
\item[\textup{(i)}] $S\subseteq V(P_{ab})$;
\item[\textup{(ii)}] $c(p_i)=c(a)+i-1$ for $1\le i\le m$; in particular
$c(b)=c(a)+m-1$ and $c\bigl(V(P_{ab})\bigr)=[\,c(a),c(b)\,]$;
\item[\textup{(iii)}] every $z\in V(T)\setminus V(P_{ab})$ satisfies
$c(z)<c(a)$ or $c(z)>c(b)$.
\end{enumerate}
\end{theorem}

\begin{proof}
\textbf{Step 1: proof of (i).}
Suppose there is $x\in S\setminus V(P_{ab})$, and let $w=\med(a,b,x)$. Since
$w\in V(P_{ab})$ and $x\notin V(P_{ab})$ we have $w\ne x$. We distinguish
three cases.

If $w\notin\{a,b\}$, then $w$ is an interior vertex of $P_{ab}$, and $a$, $b$,
$x$ lie in three distinct components of $T-w$. As $a,b,x\in S$, this
contradicts Lemma~\ref{lem:threebranch}.

If $w=a$, then $a$ lies on $P_{xb}$ and $a\notin\{x,b\}$, so $a$ is an interior
vertex of $P_{xb}$. Since $x,b\in S$, Lemma~\ref{lem:monotone} applied to $x$
and $b$ shows that $c(a)$ lies strictly between $c(x)$ and $c(b)$. This
contradicts the minimality of $c(a)$ on $S$, because $x,b\in S$ forces
$c(a)<c(x)$ and $c(a)<c(b)$.

If $w=b$, the symmetric argument shows that $c(b)$ lies strictly between $c(x)$
and $c(a)$, contradicting the maximality of $c(b)$ on $S$. Hence
$S\subseteq V(P_{ab})$.

\medskip
\noindent\textbf{Step 2: the labels increase along $P_{ab}$.}
The path $P_{ab}$ is the unique, hence shortest, $a$--$b$ path in $T$, and
$a,b\in S$ with $c(a)<c(b)$. Lemma~\ref{lem:monotone} therefore gives
\begin{equation}
  c(a)=c(p_1)<c(p_2)<\cdots<c(p_m)=c(b),
  \label{eq:increase}
\end{equation}
and in particular $c(a)\le c(p_i)\le c(b)$ for all $i$.

\medskip
\noindent\textbf{Step 3: proof of (iii).}
Suppose $z\notin V(P_{ab})$ satisfies $c(a)<c(z)<c(b)$, and let
$p_i=\med(a,b,z)$, the vertex of $P_{ab}$ nearest to $z$; put
$r=d(p_i,z)\ge1$. Every $a$--$z$ path passes through $p_i$, and likewise for
$b$, so
\[
  d(a,z)=(i-1)+r>i-1=d(a,p_i),
  \qquad
  d(b,z)=(m-i)+r>m-i=d(b,p_i).
\]

Consistency of $a$ applied to the pair $(p_i,z)$ gives $c(a,p_i)\le c(a,z)$. By
\eqref{eq:increase} we have $c(p_i)\ge c(a)$, and by assumption $c(z)>c(a)$, so
both absolute values may be resolved and we obtain
\[
  c(p_i)-c(a)\le c(z)-c(a),
  \qquad\text{that is,}\qquad c(p_i)\le c(z).
\]

Consistency of $b$ applied to the same pair gives $c(b,p_i)\le c(b,z)$. By
\eqref{eq:increase} we have $c(p_i)\le c(b)$, and by assumption $c(z)<c(b)$,
whence
\[
  c(b)-c(p_i)\le c(b)-c(z),
  \qquad\text{that is,}\qquad c(p_i)\ge c(z).
\]

Combining the two inequalities just obtained yields $c(p_i)=c(z)$. But
$p_i\in V(P_{ab})$ and $z\notin V(P_{ab})$, so $p_i\ne z$, contradicting the
injectivity of $c$. This proves (iii).

\medskip
\noindent\textbf{Step 4: proof of (ii).}
By Step 3, every vertex whose label lies in the interval $[\,c(a),c(b)\,]$
belongs to $V(P_{ab})$; hence $[\,c(a),c(b)\,]\subseteq c\bigl(V(P_{ab})\bigr)$.
By \eqref{eq:increase} the reverse inclusion also holds, so
$c\bigl(V(P_{ab})\bigr)=[\,c(a),c(b)\,]$. Comparing cardinalities gives
$c(b)-c(a)+1=m$, and \eqref{eq:increase} then forces $c(p_i)=c(a)+i-1$ for
every $i$.
\end{proof}

\begin{remark}
\label{rem:sharp}
Theorem~\ref{thm:path} concerns an arbitrary labeling, not merely an optimal
one. Notice also that it does not assert that $S$ is an interval of $P_{ab}$.
\end{remark}

Since $\ldc(T)\ge2$ for every tree $T$ by Proposition~\ref{prop:basic}(i), the
hypothesis $|S|\ge2$ of Theorem~\ref{thm:path} is satisfied by every optimal
labeling of a tree.

\subsection{Two consequences}

For the rest of this section we keep the notation of Theorem~\ref{thm:path},
and we set
\[
\begin{aligned}
  L&=\{\,z\in V(T)\setminus V(P_{ab}) : c(z)<c(a)\,\},\\
  H&=\{\,z\in V(T)\setminus V(P_{ab}) : c(z)>c(b)\,\}.
\end{aligned}
\]
so that $V(T)$ is the disjoint union of $V(P_{ab})$, $L$ and $H$ by
Theorem~\ref{thm:path}(iii). Since $V(P_{ab})$ carries exactly the labels
$[\,c(a),c(b)\,]$ by Theorem~\ref{thm:path}(ii), we have moreover
\[
  c(L)=[\,1,\ c(a)-1\,],\qquad c(H)=[\,c(b)+1,\ n\,];
\]
that is, $L$ and $H$ do not merely lie below and above the label block of
$P_{ab}$, they fill it out completely.

The first consequence bounds how fast the balls around $a$ and $b$ may grow.
It is the workhorse of Sections~\ref{sec:kary} and \ref{sec:spiders}: a tree
whose balls grow quickly simply cannot carry many consistent vertices.

\begin{corollary}
\label{cor:ballgrowth}
With the notation above, let $t$ be an integer with $1\le t\le m-2$. Then
\[
  \Ball{a}{t}\subseteq\{p_1,\dots,p_{t+1}\}\cup c^{-1}\bigl([\,c(a)-t-1,\ c(a)-1\,]\bigr),
\]
and consequently $|\Ball{a}{t}|\le 2t+2$. Symmetrically,
$\Ball{b}{t}\subseteq\{p_m,\dots,p_{m-t}\}\cup c^{-1}([\,c(b)+1,\ c(b)+t+1\,])$
and $|\Ball{b}{t}|\le 2t+2$.
\end{corollary}

\begin{proof}
Since $t+2\le m$, the vertex $p_{t+2}$ exists, and $d(a,p_{t+2})=t+1>t$, so
$p_{t+2}\notin\Ball{a}{t}$; moreover $c(a,p_{t+2})=t+1$ by
Theorem~\ref{thm:path}(ii). Lemma~\ref{lem:ball} applied to $a$ with radius $t$
therefore gives $c(a,x)\le t+1$, that is
$c(x)\in[\,c(a)-t-1,\ c(a)+t+1\,]$, for every $x\in\Ball{a}{t}$. By
Theorem~\ref{thm:path}(ii) the vertices carrying the labels
$c(a)+1,\dots,c(a)+t+1$ are exactly $p_2,\dots,p_{t+2}$, and $p_{t+2}$ has just
been excluded. This proves the stated inclusion, and the cardinality bound
follows since the two sets on the right have at most $t+1$ elements each. The
statement for $b$ follows by applying the above to the reversed labeling
$c'(x)=n+1-c(x)$ of Proposition~\ref{prop:basic}(ii), under which $a$ and $b$
exchange roles, $L$ and $H$ are interchanged, and $P_{ab}$ is traversed in the
opposite direction.
\end{proof}

The second consequence says that a branch hanging off an interior vertex of
$P_{ab}$ must be labeled entirely below $c(a)$ if that vertex is near $a$, and
entirely above $c(b)$ if it is near $b$; near the middle of the path both
requirements apply, so no branch can be there at all.

\begin{corollary}
\label{cor:branch}
With the notation above, let $1<i<m$ and let $y$ be a neighbor of $p_i$ with
$y\notin V(P_{ab})$.
\begin{enumerate}
\item[\textup{(i)}] If $i\le m-2$, then $y\in L$.
\item[\textup{(ii)}] If $i\ge 3$, then $y\in H$.
\item[\textup{(iii)}] Consequently $\deg_T(p_i)=2$ whenever $3\le i\le m-2$.
\end{enumerate}
\end{corollary}

\begin{proof}
(i) As $y$ is adjacent to $p_i$ and does not lie on $P_{ab}$, we have
$d(a,y)=i$, while $d(a,p_{i+2})=i+1$; thus $y\in\Ball{a}{i}$ and
$p_{i+2}\notin\Ball{a}{i}$, the latter vertex exists because $i\le m-2$.
Lemma~\ref{lem:ball} gives $c(a,y)\le c(a,p_{i+2})=i+1$. If $y\in H$ then
$c(a,y)=c(y)-c(a)\ge c(b)+1-c(a)=m$ by Theorem~\ref{thm:path}(ii), so
$m\le i+1$, contradicting $i\le m-2$. Hence $y\in L$.

(ii) Symmetric, using $d(b,y)=m-i+1$, $d(b,p_{i-2})=m-i+2$ and
Proposition~\ref{prop:basic}(ii).

(iii) If $3\le i\le m-2$ and $p_i$ had a neighbor $y$ outside $P_{ab}$, then
$y\in L\cap H=\emptyset$ by (i) and (ii). So all neighbors of $p_i$ lie on
$P_{ab}$, and an interior vertex of a path in a tree has exactly two such
neighbors.
\end{proof}

%% file: kary.tex
\section{Complete \texorpdfstring{$k$}{k}-ary trees}
\label{sec:kary}

Let $k\ge2$ and $h\ge1$ be integers. The \emph{complete $k$-ary tree of height
$h$}, denoted $\Tkh$, is the rooted tree in which every vertex at distance less
than $h$ from the root has exactly $k$ children and every vertex at distance $h$
from the root is a leaf. The \emph{level} of a vertex is its distance to the
root, so $\Tkh$ has $k^{j}$ vertices at level $j$ for $0\le j\le h$ and
$n(\Tkh)=(k^{h+1}-1)/(k-1)$ vertices in total. Every leaf is at level $h$; the
root has degree $k$, a vertex at level $j$ with $1\le j\le h-1$ has degree
$k+1$, and a leaf has degree $1$. Since the degrees occurring in $\Tkh$ are
therefore $1$, $k$ and $k+1$, we obtain the following observation, which will be
used twice in the proof of Theorem~\ref{thm:kary}:
\[
  \Tkh \text{ has a vertex of degree } 2
  \quad\Longleftrightarrow\quad
  k=2 ,
\]
and in that case the root is the only such vertex.
We first record a lower bound, valid for all trees, which is where the value $3$
comes from.

\begin{lemma}
\label{lem:twoleaves}
Let $T$ be a tree containing a vertex $w$ that is adjacent to at least two
leaves. Then $\ldc(T)\ge3$.
\end{lemma}

\begin{proof}
Let $x_1,x_2$ be leaves adjacent to $w$ and put $D=\deg_T(w)$. If $D=2$,
then $T=P_3$ and the result follows from Proposition~\ref{prop:path}; assume
$D\ge3$. Set $M_j=|\Ball{w}{j}|$ and label
\[
  c(x_1)=1,\qquad c(w)=2,\qquad c(x_2)=3,
\]
then assign the remaining labels in nondecreasing order of distance from $w$.
Consequently
\[
  c\bigl(\Ball{w}{j}\bigr)=[\,1,M_j\,]\qquad\text{for every } j\ge1 .
\]

For $w$, the two sides of \eqref{eq:ballcond} are $M_j-2$ and $M_j-1$.
For $x_1$, they are $1$ and $2$ at radius $1$, and $M_{j-1}-1$ and
$M_{j-1}$ at radius $j\ge2$. For $x_2$, they are both $1$ at radius $1$,
while at radius $j\ge2$ they are
$\max\{2,M_{j-1}-3\}$ and $M_{j-1}-2$; here $M_{j-1}\ge D+1\ge4$.
Thus Lemma~\ref{lem:ball} makes $w,x_1,x_2$ consistent.
\end{proof}

\begin{theorem}
\label{thm:kary}
For all integers $k\ge2$ and $h\ge1$,
\[
  \ldc\bigl(\Tkh\bigr)=
  \begin{cases}
    5, & \text{if } k=2 \text{ and } h=2,\\[2pt]
    3, & \text{otherwise.}
  \end{cases}
\]
\end{theorem}

\begin{proof}
\smallskip
\textbf{Lower bounds.} If $h=1$ then $\Tkh$ is the star $K_{1,k}$ and its
centre is adjacent to $k\ge2$ leaves; if $h\ge2$ then any vertex at level $h-1$
is adjacent to $k\ge2$ leaves. In either case Lemma~\ref{lem:twoleaves} gives
$\ldc(\Tkh)\ge3$.

For $T^{2,2}$, which has $7$ vertices, label the root by $4$, its two children
by $3$ and $5$, the two children of the vertex labeled $3$ by $1$ and $2$, and
the two children of the vertex labeled $5$ by $6$ and $7$. Writing $z_i$ for
the vertex labeled $i$, a direct application of Lemma~\ref{lem:ball} gives
$S_c(T^{2,2})=\{z_2,z_3,z_4,z_5,z_6\}$. Hence $\ldc(T^{2,2})\ge5$.

\smallskip
\textbf{Upper bounds.} We must show that $|S|\le3$ when $(k,h)\ne(2,2)$, and
that $|S|\le5$ when $(k,h)=(2,2)$. Both assertions hold trivially for labelings
with $|S|\le3$, so let $c$ be a labeling of $T=\Tkh$ with $|S|\ge4$, where
$S=S_c(T)$, and adopt the notation of Theorem~\ref{thm:path} and
Corollaries~\ref{cor:ballgrowth} and \ref{cor:branch}. We derive a
contradiction in every case except $(k,h)=(2,2)$, where we obtain $|S|\le5$.
By Theorem~\ref{thm:path}(i) we have $S\subseteq V(P_{ab})$, so
\begin{equation}
  m=|V(P_{ab})|\ \ge\ |S|\ \ge\ 4 .
  \label{eq:mfour}
\end{equation}
If $h=1$ then $\diam(T)=2$ and $m\le3$, contradicting \eqref{eq:mfour}; so
$h\ge2$. By \eqref{eq:mfour} we may apply Corollary~\ref{cor:ballgrowth} with
$t=1$ and with $t=2$, obtaining
\begin{equation}
  1+\deg(a)=|\Ball{a}{1}|\le4,
  \qquad
  |\Ball{a}{2}|\le6,
  \label{eq:balls}
\end{equation}
and, by the symmetric statement of Corollary~\ref{cor:ballgrowth}, the same two
inequalities for $b$.

\smallskip
\emph{Case 1: $k\ge3$.} The degrees occurring in $T$ are $1$, $k$ and $k+1$, so
the first inequality in \eqref{eq:balls} forces $a$ to be a leaf, or else $k=3$
and $a$ to be the root. The latter is impossible: for $k=3$ and $h\ge2$ the
root satisfies $|\Ball{a}{2}|=1+3+9=13>6$. Hence $a$, and likewise $b$, is a
leaf. Now all leaves of $T$ lie at level $h$, so
$d(a,b)=2\bigl(h-\text{level of the lowest common ancestor}\bigr)$ is even and
$m=d(a,b)+1$ is odd; with \eqref{eq:mfour} this yields $m\ge5$. But then the
index $i=3$ satisfies $3\le i\le m-2$, so $\deg(p_3)=2$ by
Corollary~\ref{cor:branch}(iii), contradicting the degree observation of the
first paragraph of this section, since $k\ge3$.

\smallskip
\emph{Case 2: $k=2$.} By that same observation the root is the only vertex of
degree $2$, so Corollary~\ref{cor:branch}(iii) shows that $\{3,4,\dots,m-2\}$
contains at most one index, whence
\begin{equation}
  m\le5 .
  \label{eq:mfive}
\end{equation}
In particular $|S|\le m\le5$. This is exactly what was required when $h=2$, and
together with $\ldc(T^{2,2})\ge5$ it gives $\ldc(T^{2,2})=5$. We may therefore
assume $h\ge3$ and aim for a contradiction.

We first locate $a$ and $b$. For $v\in V(T)$ at level $j$ we compute
$|\Ball{v}{2}|$: it equals $4$ if $v$ is a leaf (the parent, the sibling and the
grandparent, the last existing because $h\ge3$), it equals $6$ if $j=h-1$ (the
$2$ children, the parent, the sibling and the grandparent, again existing
because $h\ge3$), and it is at least $1+2+4=7$ if $j\le h-2$, since then $v$ has
two children each having two children. By the second inequality in
\eqref{eq:balls}, each of $a$ and $b$ is therefore a leaf or a vertex at level
$h-1$. All leaves lie at level $h$, so in either case the level of $a$ and the
level of $b$ determine the parity of $d(a,b)$, and we distinguish two cases
accordingly.

\smallskip
\emph{Case 2a: $a$ and $b$ are both leaves, or both at level $h-1$.} Then
$d(a,b)$ is even, so $m$ is odd; by \eqref{eq:mfour} and \eqref{eq:mfive},
$m=5$, and Corollary~\ref{cor:branch}(iii) applied to $i=3$ shows that $p_3$ is
the root. Hence $a$ is at distance $2$ from the root, i.e.\ at level $2$. If
$a$ is a leaf this gives $h=2$, contrary to $h\ge3$. If $a$ is at level $h-1$
this gives $h=3$; but then Corollary~\ref{cor:ballgrowth} applies with
$t=3\le m-2$ and yields $|\Ball{a}{3}|\le8$, whereas in $T^{2,3}$ a vertex $a$
at level $2$ satisfies
\[
\begin{aligned}
  |\Ball{a}{3}|
  & =\underbrace{1}_{a}+\underbrace{3}_{\text{2 children, parent}}
    +\underbrace{2}_{\text{root, sibling}}\\
  &\quad+\underbrace{3}_{\text{2 nephews, other child of the root}}=9>8,
\end{aligned}
\]
a contradiction.

\smallskip
\emph{Case 2b: one of $a,b$ is a leaf and the other lies at level $h-1$.} Then
$d(a,b)$ is odd, so $m$ is even, and \eqref{eq:mfour} and \eqref{eq:mfive} give
$m=4$. By Proposition~\ref{prop:basic}(ii) we may assume that $a$ is the leaf,
so that $a$ lies at level $h$. Its neighbor $p_2$ on $P_{ab}$ is therefore its
parent, at level $h-1$. The vertex $p_3$ cannot be a sibling of $a$, since such
a vertex is a leaf and $p_4$ would then not exist; hence $p_3$ is the
grandparent of $a$, at level $h-2$, and $b$ is the child of $p_3$ other than
$p_2$, at level $h-1$ as required. Thus
\[
  P_{ab}:\quad
  a\ (\text{level } h)
  \ -\ p_2\ (\text{level } h-1)
  \ -\ p_3\ (\text{level } h-2)
  \ -\ b\ (\text{level } h-1),
\]
and by Theorem~\ref{thm:path}(ii) we have $c(p_i)=c(a)+i-1$ for $i=1,2,3,4$.
Moreover $|S|\ge4=m$ and $S\subseteq V(P_{ab})$ by Theorem~\ref{thm:path}(i), so
$S=V(P_{ab})$; in particular
\[
  p_3\in S .
\]

The vertex $b$ has degree $3$: its neighbors are $p_3$ and its two children
$x,y$, which are leaves not lying on $P_{ab}$. Corollary~\ref{cor:ballgrowth}
applied to $b$ with $t=1$ gives
$\Ball{b}{1}\subseteq\{p_4,p_3\}\cup c^{-1}(\{c(b)+1,c(b)+2\})$, and since
$|\Ball{b}{1}|=4$ we conclude
\begin{equation}
  \{c(x),c(y)\}=\{c(b)+1,\ c(b)+2\}.
  \label{eq:xy}
\end{equation}
Since $h\ge3$, the vertex $p_3$ lies at level $h-2\ge1$ and therefore has a
parent $w$, which is its unique neighbor off $P_{ab}$. By
Corollary~\ref{cor:branch}(ii) applied with $i=3$ we get $w\in H$, that is,
\begin{equation}
  c(w)>c(b).
  \label{eq:wabove}
\end{equation}
On the other hand $p_3\in S$ as just observed, and $a\notin\Ball{p_3}{1}$
with $c(p_3,a)=2$; Lemma~\ref{lem:ball} applied to $p_3$ with radius $1$
therefore gives $c(p_3,w)\le2$, that is $c(w)\le c(p_3)+2=c(a)+4=c(b)+1$.
Combined with \eqref{eq:wabove} this forces $c(w)=c(b)+1$, contradicting
\eqref{eq:xy}, because $d(b,w)=2$ while $d(b,x)=d(b,y)=1$, so that $w$ is
distinct from both $x$ and $y$.

\smallskip
In all cases we have reached a contradiction, so $|S|\le3$ whenever $h\ge1$ and
$(k,h)\ne(2,2)$. This completes the proof.
\end{proof}

%% file: spiders.tex
\section{Spiders}
\label{sec:spiders}

A \emph{spider} is a tree $G$ with exactly one vertex $w$ of degree greater
than $2$; that vertex is the \emph{body} of the spider, and the components of
$G-w$, each of which is a path, are its \emph{legs}. We write
$G=SP(w;\ell_1,\dots,\ell_k)$ for the spider with body $w$ and $k\ge3$ legs
having $\ell_1\le\ell_2\le\dots\le\ell_k$ vertices, and we denote by $v_{i,j}$
the vertex of the $i$th leg with $d(v_{i,j},w)=j$, so that
$1\le j\le\ell_i$ and
\[
  n=n(G)=1+\sum_{i=1}^{k}\ell_i .
\]
If $\ell_k=1$ then $G$ is the star $K_{1,k}$ and $\ldc(G)=3$ by
\cite[Proposition 4.4]{CH2025}; we therefore assume $\ell_k\ge2$ throughout,
unless stated otherwise.

\subsection{The natural labeling}

Fix $i$ with $1\le i\le k$. The \emph{natural labeling relative to the $i$th
leg} is the labeling $c_i$ of $G$ defined by
\[
  c_i(v_{i,j})=\ell_i-j+1 \quad(1\le j\le\ell_i),
  \qquad
  c_i(w)=\ell_i+1,
\]
the remaining vertices --- those of the legs other than the $i$th --- receiving
the labels $\ell_i+2,\dots,n$ in any order such that $c_i(x)<c_i(y)$ whenever
$d(w,x)<d(w,y)$. Thus $c_i$ enumerates the $i$th leg backwards from its leaf,
then $w$, and then sweeps the remaining legs level by level away from $w$.

For $s\ge0$ we put
\[
  N^{(i)}_s=\sum_{i'\ne i}\min\{\ell_{i'},s\},
\]
the number of vertices outside the $i$th leg at distance at most $s$ from $w$;
so $N^{(i)}_0=0$ and $N^{(i)}_s\ge k-1$ for $s\ge1$. The following lemma
determines completely which vertices of the $i$th leg are consistent under
$c_i$. It replaces, in a single statement, the several separate computations
that this kind of argument would otherwise require.

\begin{lemma}
\label{lem:natural}
Let $G=SP(w;\ell_1,\dots,\ell_k)$ with $k\ge3$, let $1\le i\le k$ and let
$1\le j\le\ell_i$. Then $v_{i,j}\in S_{c_i}(G)$ if and only if
\begin{equation}
  N^{(i)}_s\le s+1
  \qquad\text{for every } s \text{ with } 1\le s\le \ell_i-2j-1 .
  \label{eq:natcond}
\end{equation}
\end{lemma}

\begin{proof}
Write $u=v_{i,j}$, $\lambda=c_i(u)=\ell_i-j+1$ and $N_s=N^{(i)}_s$, and note
that $c_i(w)=\ell_i+1=\lambda+j$.

\smallskip
\noindent\emph{The labels of the balls about $u$.}
Moving along the $i$th leg away from $w$ decreases the label by $1$ at each step
until the label $1$ is reached, and moving towards $w$ increases it by $1$ at
each step until $w$ is reached; thereafter the ball absorbs the remaining legs
level by level, and by the definition of $c_i$ these vertices carry the next
labels in increasing order. Hence $c_i(\Ball{u}{r})=[\alpha_r,\beta_r]$, where
\[
  \alpha_r=\max\{1,\ \lambda-r\},
  \qquad
  \beta_r=\begin{cases}
    \lambda+r, & 0\le r\le j-1,\\
    \lambda+j+N_{r-j}, & r\ge j .
  \end{cases}
\]
The label distances from $u$ to the vertices of $\Ball{u}{r}$ are therefore the
integers between $0$ and $\max\{\lambda-\alpha_r,\ \beta_r-\lambda\}$, while the
vertices outside $\Ball{u}{r}$ carry the labels
$[1,\alpha_r-1]\cup[\beta_r+1,n]$, whose label distances to $u$ are at least
$\lambda-\alpha_r+1$ and $\beta_r+1-\lambda$ respectively. By
Lemma~\ref{lem:ball}, $u$ is consistent if and only if for every $r\ge0$
\begin{align}
  \beta_r-\lambda &\le \lambda-\alpha_r+1 &&\text{whenever } \alpha_r\ge2,
  \label{eq:condA}\\
  \lambda-\alpha_r &\le \beta_r+1-\lambda &&\text{whenever } \beta_r\le n-1 .
  \label{eq:condB}
\end{align}

\smallskip
\noindent\emph{Condition \eqref{eq:condB} always holds.}
If $r\le j-1$ then $\lambda-\alpha_r=\min\{r,\lambda-1\}\le r\le
\beta_r+1-\lambda$. Let $r\ge j$ and put $s=r-j$, so that
$\beta_r-\lambda=j+N_s$. Assume $\beta_r\le n-1$; then some leg other than the
$i$th contains a vertex at distance more than $s$ from $w$, so that leg
contributes exactly $s$ to $N_s$ and
\begin{equation}
  N_s\ge s .
  \label{eq:Nsges}
\end{equation}
If $\lambda-\alpha_r=r=j+s$, then \eqref{eq:condB} reads $j+s\le j+N_s+1$, which
holds by \eqref{eq:Nsges}. Otherwise $\alpha_r=1$ and
$\lambda-\alpha_r=\lambda-1=\ell_i-j$, and \eqref{eq:condB} reads
$\ell_i\le 2j+N_s+1$. Since $\alpha_r=1$ we have $r\ge \lambda-1$, that is
$s\ge \ell_i-2j$. If $\ell_i\le 2j$ this gives $\ell_i\le 2j+N_s+1$ at once;
and if $\ell_i>2j$ then $s\ge \ell_i-2j\ge1$, so $N_s\ge s\ge\ell_i-2j$ by
\eqref{eq:Nsges} and again $2j+N_s+1\ge \ell_i+1>\ell_i$.

\smallskip
\noindent\emph{Condition \eqref{eq:condA} is equivalent to \eqref{eq:natcond}.}
If $r\le j-1$ then $\beta_r-\lambda=r$ and $\lambda-\alpha_r+1=r+1$ (as
$\alpha_r\ge2$ forces $\alpha_r=\lambda-r$), so \eqref{eq:condA} holds. Let
$r\ge j$, put $s=r-j$ and assume $\alpha_r\ge2$, that is $r\le \lambda-2$, that
is $0\le s\le \ell_i-2j-1$. Then $\alpha_r=\lambda-r$ and \eqref{eq:condA}
reads $j+N_s\le r+1=j+s+1$, that is $N_s\le s+1$. For $s=0$ this is trivially
true, so \eqref{eq:condA} for all $r$ is exactly \eqref{eq:natcond}.
\end{proof}

Because $N^{(i)}_s\ge k-1$ for $s\ge1$, and because
$N^{(i)}_s=\min\{\ell_{i'},s\}+\min\{\ell_{i''},s\}$ when $k=3$ and
$\{i',i''\}=\{1,2,3\}\setminus\{i\}$, condition \eqref{eq:natcond} depends on
the spider only through the following quantity.

\begin{definition}
\label{def:sigma}
For $1\le i\le k$ put
$\sigma_i=\min\{\,s\ge1 : N^{(i)}_s\ge s+2\,\}$, with $\sigma_i=\infty$ if no
such $s$ exists.
\end{definition}

\begin{lemma}
\label{lem:sigma}
Let $G=SP(w;\ell_1,\dots,\ell_k)$ with $k\ge3$. Then
\[
  \sigma_i=
  \begin{cases}
    1, & \text{if } k\ge4,\\
    \infty, & \text{if } k=3 \text{ and } \ell_{i'}=1 \text{ for some } i'\ne i,\\
    2, & \text{if } k=3 \text{ and } \ell_{i'}\ge2 \text{ for both } i'\ne i .
  \end{cases}
\]
Moreover
\begin{enumerate}
\item[\textup{(i)}] for $1\le j\le\ell_i$ one has $v_{i,j}\in S_{c_i}(G)$ if and
only if $j\ge\bigl\lceil(\ell_i-\sigma_i)/2\bigr\rceil$, so that
\[
  \bigl|S_{c_i}(G)\cap\{v_{i,1},\dots,v_{i,\ell_i}\}\bigr|
  =\min\Bigl\{\ell_i,\ \bigl\lfloor(\ell_i+\sigma_i)/2\bigr\rfloor+1\Bigr\};
\]
\item[\textup{(ii)}] $w\in S_{c_i}(G)$ if and only if $\ell_i\le\sigma_i$;
\item[\textup{(iii)}] consequently
\[
  \bigl|S_{c_i}(G)\cap\bigl(\{w\}\cup\{v_{i,1},\dots,v_{i,\ell_i}\}\bigr)\bigr|=
  \begin{cases}
    \lceil \ell_i/2\rceil+1, & \text{if } \sigma_i=1,\\
    \lfloor \ell_i/2\rfloor+2, & \text{if } \sigma_i=2,\\
    \ell_i+1, & \text{if } \sigma_i=\infty .
  \end{cases}
\]
\end{enumerate}
Here the conventions for $\sigma_i=\infty$ are that the condition in
\textup{(i)} holds for every $j$, that the count in \textup{(i)} equals
$\ell_i$, and that the inequality in \textup{(ii)} holds.
\end{lemma}

\begin{proof}
If $k\ge4$ then $N^{(i)}_1=k-1\ge3=1+2$, so $\sigma_i=1$. Let $k=3$ and
$\{i',i''\}=\{1,2,3\}\setminus\{i\}$. If $\ell_{i'}=1$ then
$N^{(i)}_s=1+\min\{\ell_{i''},s\}\le s+1$ for every $s\ge1$, so
$\sigma_i=\infty$. If $\ell_{i'},\ell_{i''}\ge2$ then $N^{(i)}_1=2<3$ while
$N^{(i)}_2=4\ge4$, so $\sigma_i=2$.

\smallskip
\noindent(i) Condition \eqref{eq:natcond} says that no $s$ with
$1\le s\le \ell_i-2j-1$ satisfies $N^{(i)}_s\ge s+2$. When $\sigma_i=\infty$
this is vacuous, and otherwise, since $\sigma_i\le2$, it is equivalent to
$\ell_i-2j-1<\sigma_i$, that is to $j\ge\lceil(\ell_i-\sigma_i)/2\rceil$. The
admissible $j$ are therefore those with
$\max\{1,\lceil(\ell_i-\sigma_i)/2\rceil\}\le j\le\ell_i$, and since
$\ell_i-\lceil(\ell_i-\sigma_i)/2\rceil=\lfloor(\ell_i+\sigma_i)/2\rfloor$ their
number is $\min\{\ell_i,\lfloor(\ell_i+\sigma_i)/2\rfloor+1\}$, as claimed.

\smallskip
\noindent The truncation in (i) is not vacuous: for $G=SP(w;2,2,2)$ and any $i$
we have $\sigma_i=2=\ell_i$, and the $i$th leg contains only two vertices, both
of them consistent, whereas $\lfloor(\ell_i+\sigma_i)/2\rfloor+1=3$.

\smallskip
\noindent(ii) Put $\lambda=c_i(w)=\ell_i+1$. For $r\ge0$ we have
$\Ball{w}{r}=\{w\}\cup\{v_{i',t}:1\le i'\le k,\ 1\le t\le\min\{\ell_{i'},r\}\}$,
whose labels run from $\max\{1,\ell_i-r+1\}$ down the $i$th leg and up to
$\ell_i+1+N^{(i)}_r$ on the remaining legs; thus
$c_i(\Ball{w}{r})=[\alpha_r,\beta_r]$ with $\alpha_r=\max\{1,\lambda-r\}$ and
$\beta_r=\lambda+N^{(i)}_r$. Exactly as in the proof of
Lemma~\ref{lem:natural}, condition \eqref{eq:condB} holds always: if
$\beta_r\le n-1$ then some leg other than the $i$th has more than $r$ vertices
and contributes $r$ to $N^{(i)}_r$, so
$\lambda-\alpha_r=\min\{r,\ell_i\}\le r\le N^{(i)}_r+1=\beta_r+1-\lambda$.
Condition \eqref{eq:condA} is required precisely when $\alpha_r\ge2$, that is
when $r\le\ell_i-1$, and then reads $N^{(i)}_r\le r+1$. Hence $w$ is consistent
if and only if no $r$ with $1\le r\le\ell_i-1$ satisfies $N^{(i)}_r\ge r+2$,
that is if and only if $\ell_i-1<\sigma_i$.

\smallskip
\noindent(iii) Add (i) and (ii) in each of the three cases. The resulting
values are
\[
\begin{array}{c@{\quad}l}
\sigma_i=1 &
  \min\{\ell_i,\lceil\ell_i/2\rceil+1\}+[\ell_i=1]
    =\lceil\ell_i/2\rceil+1,\\
\sigma_i=2 &
  \min\{\ell_i,\lfloor\ell_i/2\rfloor+2\}+[\ell_i\le2]
    =\lfloor\ell_i/2\rfloor+2,\\
\sigma_i=\infty & \ell_i+1.
\end{array}
\]
\end{proof}

Lemma~\ref{lem:sigma} accounts for the $i$th leg and for the body. When some
leg consists of a single vertex, that vertex may be consistent as well, and we
record now when this happens; the statement will be used both for $k=3$ in
Section~\ref{sec:sp3short} and for $k\ge4$ in Section~\ref{sec:spexc}.

Let $i_0$ be a leg with $\ell_{i_0}=1$ and let $i_1\ne i_0$. Among the natural
labelings relative to the leg $i_1$ we single out the one, written
$c_{i_1}^{\,i_0}$, whose tie-breaking order among the remaining legs puts
$v_{i_0,1}$ first, so that
\[
  c_{i_1}^{\,i_0}(v_{i_0,1})=\ell_{i_1}+2 .
\]
Finally we put, for $t\ge0$,
\begin{equation}
  \Ntwo_t=\sum_{i\notin\{i_0,i_1\}}\min\{\ell_i,t\} ,
  \label{eq:Mt}
\end{equation}
so that $\Ntwo_t$ counts the vertices at distance at most $t$ from $w$ lying on
legs other than $i_0$ and $i_1$; in particular $\Ntwo_1=k-2$. Note that $\Ntwo_t$ is
the analogue of $N^{(i)}_t$ with two legs removed instead of one.

\begin{lemma}
\label{lem:leafconsistent}
With the notation above, and for any $k\ge3$, we have
$v_{i_0,1}\in S_c(G)$ for $c=c_{i_1}^{\,i_0}$ if and only if
\begin{equation}
  \Ntwo_t\le t+2 \qquad\text{for every } t \text{ with } 1\le t\le \ell_{i_1}-1 .
  \label{eq:leafcond}
\end{equation}
\end{lemma}

\begin{proof}
Write $u=v_{i_0,1}$ and $\lambda=c(u)=\ell_{i_1}+2$, so that
$c(w)=\lambda-1$ and $c(v_{i_1,j})=\lambda-1-j$ for $1\le j\le\ell_{i_1}$.
Since $\ell_{i_0}=1$, the only neighbor of $u$ is $w$, whence
$\Ball{u}{r}=\{u\}\cup \Ball{w}{r-1}$ for $r\ge1$. The vertices of
$\Ball{w}{r-1}$ on the leg $i_1$ carry the labels
$\lambda-2,\lambda-3,\dots$ down to $\max\{1,\lambda-r\}$, and those on the
remaining legs carry, by the choice of tie-breaking and the definition of the
natural labeling, the labels $\lambda,\lambda+1,\dots,\lambda+\Ntwo_{r-1}$. Hence
$c(\Ball{u}{r})=[\alpha_r,\beta_r]$ with
\[
  \alpha_r=\max\{1,\ \lambda-r\},\qquad \beta_r=\lambda+\Ntwo_{r-1}.
\]
Exactly as in the proof of Lemma~\ref{lem:natural}, Lemma~\ref{lem:ball} says
that $u$ is consistent if and only if \eqref{eq:condA} and \eqref{eq:condB}
hold for every $r\ge0$, with these values of $\alpha_r$ and $\beta_r$.

Condition \eqref{eq:condB} always holds. Indeed
$\lambda-\alpha_r=\min\{r,\lambda-1\}$ and $\beta_r+1-\lambda=\Ntwo_{r-1}+1$. If
$\beta_r\le n-1$ then some leg $i\notin\{i_0,i_1\}$ contains a vertex at
distance more than $r-1$ from $w$, so $\ell_i>r-1$ and that leg contributes
$r-1$ to $\Ntwo_{r-1}$; hence $\Ntwo_{r-1}+1\ge r\ge\min\{r,\lambda-1\}$.

Condition \eqref{eq:condA} reads $\Ntwo_{r-1}\le r+1$ and is required exactly when
$\alpha_r=\lambda-r\ge2$, that is $r\le \lambda-2=\ell_{i_1}$. Setting
$t=r-1$ and noting that $\Ntwo_0=0$ makes the case $t=0$ vacuous, this is
\eqref{eq:leafcond}.
\end{proof}

\begin{corollary}
\label{cor:spiderlower}
Let $G=SP(w;\ell_1,\dots,\ell_k)$ with $k\ge3$ and $\ell_k\ge2$. Then
\[
  \ldc(G)\ \ge\
  \begin{cases}
    \lceil \ell_k/2\rceil+1, & \text{if } k\ge4,\\
    \lfloor \ell_k/2\rfloor+2, & \text{if } k=3 \text{ and } \ell_1\ge2 .
  \end{cases}
\]
\end{corollary}

\begin{proof}
Apply Lemma~\ref{lem:sigma}(iii) with $i=k$: in the first case $\sigma_k=1$ and
in the second $\sigma_k=2$, since then both legs other than the $k$th have at
least two vertices. Note that the body $w$ is needed here when
$\sigma_k=2=\ell_k$, that is for $G=SP(w;2,2,2)$, where the $k$th leg alone
supplies only two consistent vertices.
\end{proof}

\subsection{Where the path can lie}

We now turn to upper bounds. Throughout this subsection $c$ is a labeling of
$G=SP(w;\ell_1,\dots,\ell_k)$ with $|S|\ge2$, where $S=S_c(G)$, and we keep the
notation $a$, $b$, $P_{ab}=p_1\cdots p_m$ of Theorem~\ref{thm:path}, so that
$S\subseteq V(P_{ab})$ and $|S|\le m$.

\begin{lemma}
\label{lem:bodypos}
\begin{enumerate}
\item[\textup{(a)}] Every path $P$ of $G$ is of exactly one of the following
three kinds: a subpath of a single leg; the body $w$ together with an initial
segment $v_{i,1},\dots,v_{i,r}$ of one leg; or an initial segment of a leg
$i_0$, followed by $w$, followed by an initial segment of a second leg
$i_1\ne i_0$.
\item[\textup{(b)}] If $m\ge5$ and $w\in V(P_{ab})$, then
$w\in\{p_1,p_2,p_{m-1},p_m\}$.
\end{enumerate}
\end{lemma}

\begin{proof}
(a) If $w\notin V(P)$, then $P$ avoids the only vertex joining the legs to one
another, so $P$ lies inside a single component of $G-w$, that is inside one leg.
If $w\in V(P)$, then $P-w$ has one or two components, each a path contained in
a single leg and having $w$ as a neighbor of one of its ends; a subpath of a
leg one of whose ends is adjacent to $w$ is an initial segment
$v_{i,1},\dots,v_{i,r}$. According as $P-w$ has one or two components we
obtain the second or the third kind.

(b) If $w=p_r$ with $3\le r\le m-2$, then $\deg_G(w)=2$ by
Corollary~\ref{cor:branch}(iii), contradicting $\deg_G(w)=k\ge3$.
\end{proof}

\begin{lemma}
\label{lem:bodyend}
Assume $w\in\{p_1,p_m\}$. If $k\ge4$ then $m\le2$, and if $k=3$ and
$\ell_1\ge2$ then $m\le3$.
\end{lemma}

\begin{proof}
By Proposition~\ref{prop:basic}(ii) we may assume $w=p_1=a$. Suppose first
$m\ge3$, so that Corollary~\ref{cor:ballgrowth} applies with $t=1$ and gives
$1+k=|\Ball{w}{1}|\le4$, whence $k\le3$. This proves the first assertion. If
$k=3$, $\ell_1\ge2$ and $m\ge4$, then Corollary~\ref{cor:ballgrowth} with $t=2$
gives
\[
  1+\sum_{i=1}^{3}\min\{\ell_i,2\}=|\Ball{w}{2}|\le 6 ,
\]
that is $\textstyle\sum_{i}\min\{\ell_i,2\}\le5$; but every $\ell_i\ge2$ makes this sum
equal to $6$, a contradiction.
\end{proof}

\begin{lemma}
\label{lem:oneleg}
Assume $w\notin V(P_{ab})$, so that $P_{ab}$ is a subpath of the $i$th leg,
say with vertex set $\{v_{i,\alpha},\dots,v_{i,\beta}\}$, $\alpha<\beta$ and
$m=\beta-\alpha+1$. Then
\[
  m\le
  \begin{cases}
    \lceil \ell_i/2\rceil+1, & \text{if } k\ge4,\\
    \lfloor \ell_i/2\rfloor+2, & \text{if } k=3 \text{ and } \ell_{i'}\ge2
       \text{ for both } i'\ne i .
  \end{cases}
\]
\end{lemma}

\begin{proof}
Both endpoints of $P_{ab}$ lie in $\{a,b\}$, so
Corollary~\ref{cor:ballgrowth} applies to $q=v_{i,\alpha}$, the endpoint nearer
$w$, for every $t\le m-2$. For $t\ge\alpha$ the ball $\Ball{q}{t}$ consists of
the vertices $v_{i,s}$ with $s\le\min\{\ell_i,\alpha+t\}$, the body $w$, and the
vertices of the other legs at distance at most $t-\alpha$ from $w$; hence
\[
  |\Ball{q}{t}|=\min\{\ell_i,\alpha+t\}+1+N^{(i)}_{t-\alpha}
  \qquad (t\ge\alpha).
\]

\smallskip
\noindent\emph{The case $k\ge4$.}
Suppose, for a contradiction, that $\alpha+1\le m-2$. Then $m\ge\alpha+3$ and
$\beta=\alpha+m-1\ge2\alpha+2$, so $\ell_i\ge\beta\ge2\alpha+2>2\alpha+1$.
Applying Corollary~\ref{cor:ballgrowth} with $t=\alpha+1$ and using the displayed
formula together with $N^{(i)}_1=k-1$ gives
\[
  (2\alpha+1)+1+(k-1)=|\Ball{q}{\alpha+1}|\le 2(\alpha+1)+2=2\alpha+4 ,
\]
that is $k\le3$, a contradiction. Hence $\alpha+1>m-2$, so $m\le\alpha+2$ and
$\beta=\alpha+m-1\le2\alpha+1$. Therefore
$m=\beta-\alpha+1\le\beta-\tfrac{\beta-1}{2}+1=\tfrac{\beta+3}{2}$, and since
$m$ is an integer, $m\le\lceil\beta/2\rceil+1\le\lceil\ell_i/2\rceil+1$.

\smallskip
\noindent\emph{The case $k=3$ with $\ell_{i'}\ge2$ for both $i'\ne i$.}
Here $N^{(i)}_2=4$. Suppose, for a contradiction, that $\alpha+2\le m-2$.
Then $m\ge\alpha+4$ and $\beta\ge2\alpha+3$, so $\ell_i\ge2\alpha+3>2\alpha+2$,
and Corollary~\ref{cor:ballgrowth} with $t=\alpha+2$ gives
\[
  (2\alpha+2)+1+4=|\Ball{q}{\alpha+2}|\le2(\alpha+2)+2=2\alpha+6 ,
\]
that is $2\alpha+7\le2\alpha+6$, a contradiction. Hence $m\le\alpha+3$ and
$\beta\le2\alpha+2$, so $m=\beta-\alpha+1\le\tfrac{\beta+4}{2}$ and
$m\le\lfloor\beta/2\rfloor+2\le\lfloor\ell_i/2\rfloor+2$.
\end{proof}

\subsection{Paths meeting the body next to an end}

We now analyse the third kind of path in Lemma~\ref{lem:bodypos}(a) in the case
where one of the two initial segments has length one, which is where the real
difficulty lies. Throughout this subsection we assume that $w=p_2$; by
Proposition~\ref{prop:basic}(ii) the case $w=p_{m-1}$ is obtained from it by
reversing the labeling. Thus, by Lemma~\ref{lem:bodypos}(a),
\begin{equation}
  a=p_1=v_{i_0,1},\qquad w=p_2,\qquad p_{r}=v_{i_1,r-2}\ \ (3\le r\le m)
  \label{eq:B2}
\end{equation}
for two distinct legs $i_0\ne i_1$, and by Theorem~\ref{thm:path}(ii), writing
$s=c(a)$,
\begin{equation}
  c(w)=s+1,\qquad c(v_{i_1,r})=s+r+1\ \ (1\le r\le m-2),
  \label{eq:B2labels}
\end{equation}
while every vertex outside $V(P_{ab})$ has label smaller than $s$ or larger
than $s+m-1$.

The next lemma is the key to the whole section. It says that the configuration
\eqref{eq:B2} is essentially impossible unless the leg $i_0$ consists of a
single vertex --- which is exactly the situation in which the exceptional values
of $\ldc$ occur.

\begin{lemma}
\label{lem:B2}
Assume \eqref{eq:B2} and $\ell_{i_0}\ge2$. Then $m\le3$. If moreover
$\ell_{i_1}\ge2$, then $w\notin S$ and hence $|S|\le2$.
\end{lemma}

\begin{proof}
\textbf{Step 1: if $m\ge3$ then $c(v_{i_0,2})=s-1$.}
The vertex $v_{i_0,2}$ exists by hypothesis, is adjacent to $a$ and does not lie
on $P_{ab}$, so $\Ball{a}{1}=\{a,w,v_{i_0,2}\}$. The vertex $p_3=v_{i_1,1}$
satisfies $d(a,p_3)=2$ and, by \eqref{eq:B2labels}, $c(a,p_3)=2$. Applying
Lemma~\ref{lem:ball} to $a$ with radius $1$ gives $c(a,v_{i_0,2})\le
c(a,p_3)=2$. Since $c(v_{i_0,2})\notin[s,s+m-1]$ and $m\ge3$, the values
$s+1$ and $s+2$ are excluded, so $c(v_{i_0,2})\in\{s-2,s-1\}$; as labels are
positive this already forces $s\ge2$.

Suppose $c(v_{i_0,2})=s-2$, so that $s\ge3$. Then $c(a,v_{i_0,2})=2$, and
Lemma~\ref{lem:ball} at radius $1$ forces every vertex outside $\Ball{a}{1}$ to
have label distance at least $2$ from $a$. But $c(\Ball{a}{1})=\{s-2,s,s+1\}$,
so the vertex carrying the label $s-1$ lies outside $\Ball{a}{1}$ and has label
distance $1$ from $a$, a contradiction. Hence $c(v_{i_0,2})=s-1$.

\medskip
\noindent\textbf{Step 2: if $m\ge4$ then $c(v_{i'',1})\le s-2$ for every leg
$i''\notin\{i_0,i_1\}$.}
Such a leg exists because $k\ge3$. We have $d(a,v_{i'',1})=2$, so
$v_{i'',1}\in\Ball{a}{2}$, and since $m\ge4$ Corollary~\ref{cor:ballgrowth}
applies with $t=2$ and gives
$\Ball{a}{2}\subseteq\{p_1,p_2,p_3\}\cup c^{-1}([\,s-3,s-1\,])$. As
$v_{i'',1}\notin V(P_{ab})$ we get $c(v_{i'',1})\in[s-3,s-1]$, and
$c(v_{i'',1})\ne s-1=c(v_{i_0,2})$ by Step 1.

\medskip
\noindent\textbf{Step 3: if $m\ge4$ then no vertex of the leg $i_1$ is
consistent.}
Let $u=v_{i_1,j}$ with $1\le j\le m-2$ and suppose $u\in S$, so that
$c(u)=s+j+1$ by \eqref{eq:B2labels}. Fix a leg $i''\notin\{i_0,i_1\}$. Since
$d(u,w)=j$ we have
\[
  d(u,v_{i'',1})=j+1,\qquad d(u,v_{i_0,2})=j+2,
\]
so $v_{i'',1}\in\Ball{u}{j+1}$ while $v_{i_0,2}\notin\Ball{u}{j+1}$. By Steps 1
and 2,
\[
\begin{aligned}
  c(u,v_{i'',1})&=(s+j+1)-c(v_{i'',1})\ge j+3,\\
  c(u,v_{i_0,2})&=(s+j+1)-(s-1)=j+2,
\end{aligned}
\]
which contradicts Lemma~\ref{lem:ball} applied to $u$ with radius $j+1$. Since
$b=p_m=v_{i_1,m-2}$ belongs to $S$, this shows that $m\ge4$ is impossible;
hence $m\le3$.

\medskip
\noindent\textbf{Step 4: the case $m=3$ with $\ell_{i_1}\ge2$.}
Here $P_{ab}=a\,w\,b$ with $a=v_{i_0,1}$, $b=v_{i_1,1}$ and $c(b)=s+2$. By
Step 1 we have $c(v_{i_0,2})=s-1$. Reversing the labeling interchanges $a$ with
$b$ and $i_0$ with $i_1$, and replaces $c$ by $c'(x)=n+1-c(x)$, so
$c'(a)=n-s-1$; Step 1 applied to $c'$ gives $c'(v_{i_1,2})=c'(b)-1$, that is
$c(v_{i_1,2})=c(b)+1=s+3$.

Suppose $w\in S$. Then $\Ball{w}{1}=\{w\}\cup\{v_{i',1}:1\le i'\le k\}$, while
$v_{i_0,2}\notin\Ball{w}{1}$ has $c(w,v_{i_0,2})=2$. Lemma~\ref{lem:ball} at
radius $1$ therefore forces $c(w,v_{i',1})\le2$, that is
$c(v_{i',1})\in[\,s-1,s+3\,]$, for every $i'$. But the five labels
$s-1,s,s+1,s+2,s+3$ are carried by $v_{i_0,2},a,w,b,v_{i_1,2}$ respectively, and
$k\ge3$ provides a leg $i''\notin\{i_0,i_1\}$ whose vertex $v_{i'',1}$ is none
of these --- a contradiction. Hence $w\notin S$ and $S\subseteq\{a,b\}$.
\end{proof}

When $\ell_{i_0}=1$ the first step of the previous proof breaks down, because
$a$ is then a leaf and $\Ball{a}{1}=\{a,w\}$ carries no information. This is
precisely the source of the exceptional values of $\ldc$. For $k\ge4$ we can
nevertheless bound $|S|$; note that in the situation of the next lemma
$\ell_{i_0}=1$ holds automatically, by Lemma~\ref{lem:B2}.

\begin{lemma}
\label{lem:B2short}
Assume \eqref{eq:B2}, $k\ge4$ and $m\ge4$, so that $\ell_{i_0}=1$. Then
\begin{enumerate}
\item[\textup{(i)}] $w\notin S$;
\item[\textup{(ii)}] every $j$ with $v_{i_1,j}\in S$ satisfies $2j+2>m-2$;
\item[\textup{(iii)}] $|S|\le\lceil m/2\rceil+1\le\lceil \ell_{i_1}/2\rceil+2$.
\end{enumerate}
\end{lemma}

\begin{proof}
(i) Suppose $w\in S$. Then $\Ball{w}{1}=\{w\}\cup\{v_{i',1}:1\le i'\le k\}$,
while $p_4=v_{i_1,2}\notin\Ball{w}{1}$ has $c(w,p_4)=2$ by
\eqref{eq:B2labels}. Lemma~\ref{lem:ball} at radius $1$ gives
$c(v_{i',1})\in[\,s-1,s+3\,]$ for every $i'$. The labels $s,s+1,s+2$ are
carried by $a$, $w$ and $v_{i_1,1}$, and $s+3$ by $p_4\notin\Ball{w}{1}$; so the
$k-2\ge2$ vertices $v_{i'',1}$ with $i''\notin\{i_0,i_1\}$ would all have to
carry the single label $s-1$, which is absurd.

\medskip
\noindent(ii) Suppose $v_{i_1,j}\in S$ and $2j+2\le m-2$; write $u=v_{i_1,j}$,
so $c(u)=s+j+1$. The $k-2\ge2$ vertices $v_{i'',1}$ with
$i''\notin\{i_0,i_1\}$ lie at distance $j+1$ from $u$ and off $P_{ab}$, so each
of them has label smaller than $s$ or larger than $s+m-1$. We claim that one of
them satisfies $c(u,v_{i'',1})\ge j+3$. Indeed, if two of them have labels
smaller than $s$, these labels are distinct, so one of them is at most $s-2$ and
the corresponding label distance is at least $j+3$; and if one of them has label
larger than $s+m-1$, its label distance exceeds $m-j-2\ge(2j+4)-j-2=j+2$, hence
is at least $j+3$. On the other hand $v_{i_1,2j+2}$ lies on $P_{ab}$ (because
$2j+2\le m-2$), satisfies $d(u,v_{i_1,2j+2})=j+2$ and, by
\eqref{eq:B2labels}, $c(u,v_{i_1,2j+2})=j+2$. This contradicts
Lemma~\ref{lem:ball} applied to $u$ with radius $j+1$.

\medskip
\noindent(iii) By (ii) every consistent vertex of the leg $i_1$ has index
$j\ge\lceil (m-3)/2\rceil$, and the indices occurring on $P_{ab}$ are at most
$m-2$; counting gives
$|S\cap\{v_{i_1,j}\}_{j}|\le\lceil m/2\rceil$. Since $S\subseteq V(P_{ab})$ and
$w\notin S$ by (i), the only further possible member of $S$ is $a$, so
$|S|\le\lceil m/2\rceil+1$. Finally $m-2\le \ell_{i_1}$ gives
$\lceil m/2\rceil\le\lceil \ell_{i_1}/2\rceil+1$.
\end{proof}

\subsection{Spiders with no leg of length one}

\begin{theorem}
\label{thm:spiderlong}
Let $G=SP(w;\ell_1,\dots,\ell_k)$ with $k\ge3$ and $\ell_1\ge2$. Then
\[
  \ldc(G)=
  \begin{cases}
    \bigl\lfloor \ell_3/2\bigr\rfloor+2, & \text{if } k=3,\\[2pt]
    \bigl\lceil \ell_k/2\bigr\rceil+1, & \text{if } k\ge4 .
  \end{cases}
\]
\end{theorem}

\begin{proof}
The lower bounds are Corollary~\ref{cor:spiderlower}. For the upper bounds, let
$c$ be a labeling with $|S|\ge2$, where $S=S_c(G)$, and let $m=|V(P_{ab})|$, so
that $|S|\le m$ by Theorem~\ref{thm:path}(i). By
Lemma~\ref{lem:bodypos}(a) we may consider the position of $w$.

If $w=p_r$ with $3\le r\le m-2$, then $\deg_G(w)=2$ by
Corollary~\ref{cor:branch}(iii), which is false; so this does not occur.

If $w\in\{p_2,p_{m-1}\}$, then $P_{ab}$ is of the third kind in
Lemma~\ref{lem:bodypos}(a) with one initial segment of length one, so
\eqref{eq:B2} holds after reversing the labeling if necessary;
Lemma~\ref{lem:B2} applies, all legs having at least two vertices, and yields
$|S|\le2$.

If $w\in\{p_1,p_m\}$, then Lemma~\ref{lem:bodyend} gives $m\le2$ when $k\ge4$
and $m\le3$ when $k=3$; hence $|S|\le2$, respectively $|S|\le3$.

If $w\notin V(P_{ab})$, then Lemma~\ref{lem:oneleg} gives
$m\le\lceil \ell_i/2\rceil+1\le\lceil \ell_k/2\rceil+1$ when $k\ge4$, and
$m\le\lfloor \ell_i/2\rfloor+2\le\lfloor \ell_3/2\rfloor+2$ when $k=3$.

Since $\ell_k\ge2$ we have $\lceil \ell_k/2\rceil+1\ge2$ and
$\lfloor \ell_3/2\rfloor+2\ge3$, so in every case $|S|$ is bounded by the
asserted value. The last comparison is tight: for $k=3$ and $\ell_3=2$, that is
for $G=SP(w;2,2,2)$, the asserted value is $3$, which is exactly the bound
supplied by the case $w\in\{p_1,p_m\}$.
\end{proof}

\subsection{Spiders with three legs, one of length one}
\label{sec:sp3short}

\begin{theorem}
\label{thm:spider3short}
If $k=3$ and $\ell_1=1$, then $\ldc\bigl(SP(w;1,\ell_2,\ell_3)\bigr)=\ell_3+2$.
\end{theorem}

\begin{proof}
If $\ell_3=1$ the spider is the star $K_{1,3}$ and the claim is $\ldc=3$, which
holds by \cite[Proposition 4.4]{CH2025}; so let $\ell_3\ge2$.

\emph{Lower bound.} Take $i_0=1$ and $i_1=3$, and consider the labeling
$c=c_3^{\,1}$. Since $\ell_1=1$, Lemma~\ref{lem:sigma} gives $\sigma_3=\infty$,
so by parts (i) and (ii) of that lemma all $\ell_3$ vertices of the third leg
and the body $w$ are consistent under $c$. Moreover
$\Ntwo_t=\min\{\ell_2,t\}\le t\le t+2$ for every $t$, so $v_{1,1}$ is consistent by
Lemma~\ref{lem:leafconsistent}. Hence $\ldc(G)\ge\ell_3+2$.

\emph{Upper bound.} Let $c$ be a labeling with $|S|\ge2$. If
$w\notin V(P_{ab})$ then $P_{ab}$ lies in a single leg by
Lemma~\ref{lem:bodypos}(a), so $m\le\ell_3$; note that
Lemma~\ref{lem:oneleg} does not apply here, since $k=3$ and one of the legs
other than the one containing $P_{ab}$ may have length one. If
$w\in\{p_1,p_m\}$ then $m\le1+\ell_3$; if $w\in\{p_2,p_{m-1}\}$ then $m-2$
vertices of $P_{ab}$ lie on one leg, so $m\le\ell_3+2$; and $w=p_r$ with
$3\le r\le m-2$ is impossible by Corollary~\ref{cor:branch}(iii). In every case
$|S|\le m\le\ell_3+2$.
\end{proof}

\begin{proposition}
\label{prop:spiderk4}
Let $G=SP(w;1,\ell_2,\dots,\ell_k)$ with $k\ge4$ and $\ell_k\ge2$. Then
\begin{enumerate}
\item[\textup{(i)}] $\bigl\lceil \ell_k/2\bigr\rceil+1\ \le\ \ldc(G)\ \le\
\bigl\lceil \ell_k/2\bigr\rceil+2$;
\item[\textup{(ii)}] if $\ldc(G)=\lceil \ell_k/2\rceil+2$ and $\ell_k\ge3$, then
every optimal labeling $c$ satisfies \eqref{eq:B2} with $\ell_{i_0}=1$ and
$m\ge4$, and $S_c(G)$ consists of the leaf $v_{i_0,1}$ together with
$\lceil \ell_k/2\rceil+1$ vertices of the leg $i_1$.
\end{enumerate}
\end{proposition}

\begin{proof}
The lower bound in (i) is Corollary~\ref{cor:spiderlower}. For the upper bound,
argue as in the proof of Theorem~\ref{thm:spiderlong}: the position
$3\le r\le m-2$ of $w$ is impossible, $w\in\{p_1,p_m\}$ gives $|S|\le m\le2$ by
Lemma~\ref{lem:bodyend}, and $w\notin V(P_{ab})$ gives
$|S|\le m\le\lceil \ell_k/2\rceil+1$ by Lemma~\ref{lem:oneleg}. In the
remaining case $w\in\{p_2,p_{m-1}\}$ we may assume \eqref{eq:B2} after reversing
the labeling. If $m\le3$ then $|S|\le3\le\lceil \ell_k/2\rceil+2$; and if
$m\ge4$ then $\ell_{i_0}=1$ by Lemma~\ref{lem:B2}, and
Lemma~\ref{lem:B2short}(iii) gives
$|S|\le\lceil \ell_{i_1}/2\rceil+2\le\lceil \ell_k/2\rceil+2$.

For (ii), note that $\ell_k\ge3$ makes $\lceil \ell_k/2\rceil+2\ge4$, so none of
the first three cases above, nor the sub-case $m\le3$ of the fourth, can produce
equality. Hence \eqref{eq:B2} holds with $m\ge4$ and $\ell_{i_0}=1$, and
$w\notin S$ by Lemma~\ref{lem:B2short}(i), so $S\setminus\{a\}$ is contained in
the leg $i_1$.
\end{proof}

\subsection{Spiders with a leg of length one and \texorpdfstring{$k\ge4$}{k>=4}}
\label{sec:spexc}

By Proposition~\ref{prop:spiderk4}, deciding $\ldc(G)$ for $k\ge4$ and
$\ell_1=1$ amounts to deciding when a pendant leaf and
$\lceil \ell_k/2\rceil+1$ vertices of a single leg can be made consistent
simultaneously. Lemma~\ref{lem:leafconsistent} tells us exactly when the
natural labeling achieves this, and we now translate its condition into an
explicit list. Recall that we index the legs so that
$1=\ell_1\le\ell_2\le\dots\le\ell_k$.

\begin{theorem}
\label{thm:spiderexc}
Let $G=SP(w;1,\ell_2,\dots,\ell_k)$ with $k\ge4$ and $\ell_k\ge2$. If
\begin{equation}
\begin{aligned}
  &k=4 \ \text{ and either } \ \ell_2\le2 \ \text{ or } \ (\ell_2=3
     \text{ and } \ell_4\le4),\\
  \text{or }\quad
  &k=5 \ \text{ and either } \ \ell_3=1 \ \text{ or } \ \ell_5=2,\\
  \text{or }\quad
  &k\ge6, \ \ell_2=1 \ \text{ and } \ \ell_k=2,
\end{aligned}
  \label{eq:exclist}
\end{equation}
then $\ldc(G)=\lceil \ell_k/2\rceil+2$.
\end{theorem}

\begin{proof}
By Proposition~\ref{prop:spiderk4}(i) it suffices to exhibit a labeling with
$\lceil \ell_k/2\rceil+2$ consistent vertices. Take $i_0=1$ and choose a leg
$i_1\ne1$ with $\lceil \ell_{i_1}/2\rceil=\lceil \ell_k/2\rceil$; we use
$c=c_{i_1}^{\,i_0}$. By Lemma~\ref{lem:sigma} we have $\sigma_{i_1}=1$ because
$k\ge4$, so by Lemma~\ref{lem:sigma}(iii) the labeling $c$ makes
$\lceil \ell_{i_1}/2\rceil+1=\lceil \ell_k/2\rceil+1$ vertices among the leg
$i_1$ and the body $w$ consistent; by Lemma~\ref{lem:leafconsistent} it makes
$v_{1,1}$ consistent as well provided \eqref{eq:leafcond} holds. It therefore
remains to check that, for each tuple in \eqref{eq:exclist}, the leg $i_1$ can be
chosen so that \eqref{eq:leafcond} holds. We write $p\le q\le\cdots$ for the
lengths of the legs other than $i_0$ and $i_1$, so that $\Ntwo_t$ is the sum of the
$\min\{\cdot,t\}$ over these $k-2$ numbers.

\smallskip
\noindent\emph{The case $\ell_{i_1}=1$.} The range in \eqref{eq:leafcond} is
empty and the condition holds vacuously. This requires
$\lceil \ell_k/2\rceil=1$, that is $\ell_k=2$, and a second leg of length $1$,
that is $\ell_2=1$; this is the third line of \eqref{eq:exclist}, and also
covers the sub-cases $\ell_2=1,\ \ell_k=2$ of the first two lines. Here the
count $\lceil \ell_{i_1}/2\rceil+1=2$ of Lemma~\ref{lem:sigma}(iii) consists of
$v_{i_1,1}$ and $w$, the body being consistent because
$\ell_{i_1}=1\le1=\sigma_{i_1}$.

\smallskip
\noindent Now let $\ell_{i_1}\ge2$, so that $t=1$ lies in the range and
\eqref{eq:leafcond} forces $\Ntwo_1=k-2\le3$, that is $k\le5$.

\smallskip
\noindent\emph{The case $k=4$.} Here $\Ntwo_t=\min\{p,t\}+\min\{q,t\}$. The
condition at $t=1$ and $t=2$ is automatic, at $t=3$ it reads
$\min\{p,3\}+\min\{q,3\}\le5$, that is $p\le2$, and for $t\ge3$ with $p\le2$ it
holds because $p+\min\{q,t\}\le2+t$. Hence \eqref{eq:leafcond} is equivalent to
\begin{equation}
  \ell_{i_1}\le3\quad\text{or}\quad p\le2 .
  \label{eq:k4cond}
\end{equation}
If $\ell_2\le2$, take $i_1=4$; then $p=\ell_2\le2$ and \eqref{eq:k4cond} holds.
If $\ell_2=3$ and $\ell_4\le4$, take $i_1=2$; then
$\lceil \ell_2/2\rceil=2=\lceil \ell_4/2\rceil$ and $\ell_{i_1}=3$, so
\eqref{eq:k4cond} holds.

\smallskip
\noindent\emph{The case $k=5$.} Here $\Ntwo_t$ is a sum of three terms
$\min\{p,t\}$, $\min\{q,t\}$, $\min\{r,t\}$ with $p\le q\le r$. The condition
at $t=1$ is automatic; at $t=2$ it reads
$\min\{p,2\}+\min\{q,2\}+\min\{r,2\}\le4$, that is $p=q=1$; and for $t\ge2$ with
$p=q=1$ it holds because $2+\min\{r,t\}\le t+2$. Hence \eqref{eq:leafcond} is
equivalent to
\begin{equation}
  \ell_{i_1}\le2\quad\text{or}\quad p=q=1 .
  \label{eq:k5cond}
\end{equation}
If $\ell_3=1$, take $i_1=5$; then $p=\ell_2=1$ and $q=\ell_3=1$. If
$\ell_5=2$, take $i_1=5$; then $\ell_{i_1}=2$.
\end{proof}

Conversely, the consistency of the pendant leaf imposes a condition of the same
shape on every optimal labeling, not only on the natural one.

\begin{proposition}
\label{prop:excnecessary}
Let $G=SP(w;1,\ell_2,\dots,\ell_k)$ with $k\ge4$ and $\ell_k\ge2$, and suppose
$\ldc(G)=\lceil \ell_k/2\rceil+2$. Let $c$ be an optimal labeling and adopt the
notation of \eqref{eq:B2}, so that $a=v_{i_0,1}$ with $\ell_{i_0}=1$ and
$S\setminus\{a\}$ lies on the leg $i_1$. Then, with $\Ntwo_t$ as in
\eqref{eq:Mt},
\[
  \Ntwo_t\le t+2\qquad\text{for every } t \text{ with } 1\le t\le m-3 .
\]
In particular $k\le5$ whenever $m\ge4$.
\end{proposition}

\begin{proof}
Let $1\le r\le m-2$. Since $\ell_{i_0}=1$ we have
$\Ball{a}{r}=\{a\}\cup\Ball{w}{r-1}$. The vertex $v_{i_1,r}=p_{r+2}$ lies on
$P_{ab}$, satisfies $d(a,v_{i_1,r})=r+1$ and, by \eqref{eq:B2labels},
$c(a,v_{i_1,r})=r+1$; thus $v_{i_1,r}\notin\Ball{a}{r}$. Lemma~\ref{lem:ball}
applied to $a$ with radius $r$ therefore gives $c(x)\in[\,s-r-1,\ s+r+1\,]$ for
every $x\in\Ball{a}{r}$.

Now let $x$ be one of the $\Ntwo_{r-1}$ vertices of $\Ball{a}{r}$ lying on a leg
other than $i_0$ and $i_1$. Then $x\notin V(P_{ab})$, so $c(x)<s$ or
$c(x)>s+m-1$ by Theorem~\ref{thm:path}(iii); the latter is impossible because
$c(x)\le s+r+1\le s+m-1$. Hence all these $\Ntwo_{r-1}$ vertices carry distinct
labels in the interval $[\,s-r-1,\ s-1\,]$, which contains $r+1$ integers, so
$\Ntwo_{r-1}\le r+1$. Setting $t=r-1$ gives the assertion, and $t=1$ gives
$k-2=\Ntwo_1\le3$.
\end{proof}

\begin{lemma}
\label{lem:excrange}
Let $G=SP(w;1,\ell_2,\dots,\ell_k)$ with $k\ge4$ and $\ell_k\ge3$, suppose
$\ldc(G)=\lceil \ell_k/2\rceil+2$, let $c$ be an optimal labeling and adopt the
notation of \eqref{eq:B2}. Then
\begin{enumerate}
\item[\textup{(i)}] $\lceil \ell_{i_1}/2\rceil=\lceil \ell_k/2\rceil$ and
$m\in\{\ell_{i_1}+1,\ \ell_{i_1}+2\}$;
\item[\textup{(ii)}] $\Ntwo_t\le t+2$ for every $t$ with $1\le t\le \ell_{i_1}-1$.
\end{enumerate}
\end{lemma}

\begin{proof}
(i) By Proposition~\ref{prop:spiderk4}(ii) we have $w=p_2$, $\ell_{i_0}=1$ and
$m\ge4$. Lemma~\ref{lem:B2short}(iii) gives $|S|\le\lceil m/2\rceil+1$, and
$m-2\le \ell_{i_1}$ gives $\lceil m/2\rceil\le\lceil \ell_{i_1}/2\rceil+1$;
hence
\[
  \bigl\lceil \ell_k/2\bigr\rceil+2=|S|\le\bigl\lceil m/2\bigr\rceil+1
  \le\bigl\lceil \ell_{i_1}/2\bigr\rceil+2 ,
\]
so $\lceil \ell_{i_1}/2\rceil\ge\lceil \ell_k/2\rceil$, and the reverse
inequality holds because $\ell_{i_1}\le \ell_k$. Both displayed inequalities
are therefore equalities; in particular
$\lceil m/2\rceil=\lceil \ell_{i_1}/2\rceil+1$, which together with
$m\le \ell_{i_1}+2$ leaves only $m\in\{\ell_{i_1}+1,\ell_{i_1}+2\}$.

\smallskip
\noindent(ii) If $m=\ell_{i_1}+2$ then $\ell_{i_1}-1=m-3$ and the assertion is
Proposition~\ref{prop:excnecessary}. So let $m=\ell_{i_1}+1$. Then
Proposition~\ref{prop:excnecessary} covers $1\le t\le m-3=\ell_{i_1}-2$, and it
remains to treat $t=m-2$.

Since $\ell_{i_1}=m-1$, the vertex $z=v_{i_1,m-1}$ exists; it is the leaf of the
leg $i_1$, it is adjacent to $b=p_m=v_{i_1,m-2}$, and it does not lie on
$P_{ab}$. Applying Corollary~\ref{cor:ballgrowth} to $b$ with $t=1$ gives
$\Ball{b}{1}\subseteq\{p_m,p_{m-1}\}\cup c^{-1}([\,c(b)+1,c(b)+2\,])$, and since
$\Ball{b}{1}=\{b,p_{m-1},z\}$ we conclude
\[
  c(z)\in\{\,s+m,\ s+m+1\,\}.
\]

Let $X$ be the set of vertices lying on legs other than $i_0$ and $i_1$ at
distance at most $m-2$ from $w$, so that $|X|=\Ntwo_{m-2}$ and, since
$\ell_{i_0}=1$,
\[
  \Ball{a}{m-1}=\{a\}\cup\Ball{w}{m-2}
  =\{a,w\}\cup\{v_{i_1,j}:1\le j\le m-2\}\cup X ,
\]
all members of which except those of $X$ lie on $P_{ab}$. Note also
$d(a,z)=1+(m-1)=m$, so $z\notin\Ball{a}{m-1}$ and $c(a,z)=c(z)-s$.

We claim that some vertex outside $\Ball{a}{m-1}$ has label distance exactly $m$
from $a$. If $c(z)=s+m$ this is $z$ itself. Otherwise $c(z)=s+m+1$, so
$s+m+1=c(z)\le n$ and the vertex $y=c^{-1}(s+m)$ exists. As
$c(V(P_{ab}))=[\,s,s+m-1\,]$ we have $y\notin V(P_{ab})$, and $y\ne z$.
Moreover $y\notin X$: otherwise $d(b,y)=(m-2)+d(w,y)\ge m-1\ge2>1=d(b,z)$, so
consistency of $b$ would force $c(b,z)\le c(b,y)$, that is $2\le1$. Hence
$y\notin\Ball{a}{m-1}$ and $c(a,y)=m$, proving the claim.

Lemma~\ref{lem:ball} applied to $a$ with radius $m-1$ now gives $c(a,x)\le m$
for every $x\in\Ball{a}{m-1}$, that is $c(x)\in[\,s-m,\ s+m\,]$. Let $x\in X$.
Then $x\notin V(P_{ab})$, so $c(x)<s$ or $c(x)>s+m-1$ by
Theorem~\ref{thm:path}(iii); the latter would force $c(x)=s+m$, which is the
label of $z$ or of $y$, neither of which belongs to $X$. Hence every $x\in X$
satisfies $c(x)\in[\,s-m,\ s-1\,]$, an interval containing $m$ integers, so
$\Ntwo_{m-2}=|X|\le m=(m-2)+2$ as required.
\end{proof}

We can now characterise the exceptional spiders.

\begin{theorem}
\label{thm:spiderfull}
Let $G=SP(w;1,\ell_2,\dots,\ell_k)$ with $k\ge4$ and $\ell_k\ge2$. Then
$\ldc(G)=\lceil \ell_k/2\rceil+2$ if \eqref{eq:exclist} holds, and
$\ldc(G)=\lceil \ell_k/2\rceil+1$ otherwise.
\end{theorem}

\begin{proof}
Sufficiency is Theorem~\ref{thm:spiderexc}, and
Proposition~\ref{prop:spiderk4}(i) shows that no other value can occur, so it
suffices to prove that $\ldc(G)=\lceil \ell_k/2\rceil+2$ implies
\eqref{eq:exclist}. Let $c$ be an optimal labeling.

\smallskip
\noindent\emph{The case $\ell_k\ge3$.} Let $(i_0,i_1)$ be the pair arising from
$c$ as in \eqref{eq:B2}. If some leg other than $i_0$ has length one, replacing
$i_0$ by that leg leaves the multiset $\{\ell_i : i\notin\{i_0,i_1\}\}$
unchanged, so the quantity $\Ntwo_t$ of \eqref{eq:Mt} depends only on $i_1$. By
Lemma~\ref{lem:excrange} we have
$\lceil \ell_{i_1}/2\rceil=\lceil \ell_k/2\rceil\ge2$, so $\ell_{i_1}\ge2$ and
$t=1$ lies in the range of Lemma~\ref{lem:excrange}(ii); this gives
$k-2=\Ntwo_1\le3$, that is $k\le5$. Write $p\le q\le\cdots$ for the lengths of the
legs other than $i_0$ and $i_1$.

\smallskip
\noindent Let $k=4$. As shown in the proof of Theorem~\ref{thm:spiderexc},
Lemma~\ref{lem:excrange}(ii) is equivalent to \eqref{eq:k4cond}, that is to
$\ell_{i_1}\le3$ or $p\le2$. Assume $\ell_2\ge3$, since otherwise
\eqref{eq:exclist} already holds. The two legs other than $i_0$ and $i_1$ are
among the legs $2,3,4$, so $p\ge \ell_2\ge3$ and \eqref{eq:k4cond} forces
$\ell_{i_1}\le3$. As $\ell_{i_1}\ge \ell_2\ge3$ we get
$\ell_{i_1}=\ell_2=3$, and then $2=\lceil \ell_2/2\rceil=\lceil \ell_4/2\rceil$
gives $\ell_4\le4$. Thus $\ell_2=3$ and $\ell_4\le4$, which is
\eqref{eq:exclist}.

\smallskip
\noindent Let $k=5$. Similarly Lemma~\ref{lem:excrange}(ii) is equivalent to
\eqref{eq:k5cond}, that is to $\ell_{i_1}\le2$ or $p=q=1$. In the first case
$\lceil \ell_{i_1}/2\rceil=1=\lceil \ell_5/2\rceil$ gives $\ell_5=2$; in the
second case two of the three legs other than $i_0$ and $i_1$ have length one,
which together with $\ell_{i_0}=1$ gives three legs of length one, so
$\ell_3=1$. Either way \eqref{eq:exclist} holds.

\smallskip
\noindent\emph{The case $\ell_k=2$.} Here $\lceil \ell_k/2\rceil+2=3$, and
\eqref{eq:exclist} holds automatically for $k=4$, because $\ell_2\le \ell_k=2$,
and for $k=5$, because $\ell_5=2$. So let $k\ge6$ and suppose, for a
contradiction, that $\ell_2\ge2$; then every leg other than the first has length
exactly $2$, and the first leg is the unique leg of length one.

We have $|S|=3$, so $m\ge3$. By the proof of
Proposition~\ref{prop:spiderk4}(i) the cases $w\notin V(P_{ab})$ and
$w\in\{p_1,p_m\}$ give $|S|\le2$, so $w\in\{p_2,p_{m-1}\}$ and we may assume
\eqref{eq:B2}; moreover $\ell_{i_0}=1$ by Lemma~\ref{lem:B2}, so $i_0$ is the
first leg and $\ell_{i_1}=2$, whence $m\le \ell_{i_1}+2=4$.

If $m=4$, then Proposition~\ref{prop:excnecessary} applies with $t=1\le m-3$ and
gives $k-2=\Ntwo_1\le3$, contradicting $k\ge6$.

So $m=3$ and $S=\{a,w,b\}$ with $a=v_{i_0,1}$ and $b=v_{i_1,1}$. The vertex
$v_{i_1,2}$ exists, is adjacent to $b=p_3$ and lies off $P_{ab}$, so
Corollary~\ref{cor:ballgrowth} applied to $b$ with $t=1$ gives
$c(v_{i_1,2})\in\{s+3,s+4\}$; in particular $c(w,v_{i_1,2})\le3$ while
$v_{i_1,2}\notin\Ball{w}{1}$. Now
$\Ball{w}{1}=\{w,a,b\}\cup\{v_{i'',1}:i''\notin\{i_0,i_1\}\}$, and
Lemma~\ref{lem:ball} applied to $w$ with radius $1$ shows that each of the
$k-2$ vertices $v_{i'',1}$ has label distance at most $3$ from $c(w)=s+1$, hence
carries a label in
$[\,s-2,\,s+4\,]\setminus\{s,\,s+1,\,s+2,\,c(v_{i_1,2})\}$, a set of exactly
three integers. Therefore $k-2\le3$, contradicting $k\ge6$.
\end{proof}

Collecting Lemma~\ref{lem:sigma}, Theorem~\ref{thm:spiderlong},
Theorem~\ref{thm:spider3short} and Theorem~\ref{thm:spiderfull}, we obtain the
complete determination of $\ldc$ for spiders.

\begin{theorem}
\label{thm:spidersummary}
Let $G=SP(w;\ell_1,\dots,\ell_k)$ with $k\ge3$. Then
\[
  \ldc(G)=
  \begin{cases}
    3, & \text{if } \ell_k=1,\\[2pt]
    \ell_3+2, & \text{if } k=3 \text{ and } \ell_1=1,\\[2pt]
    \bigl\lfloor \ell_3/2\bigr\rfloor+2, & \text{if } k=3 \text{ and } \ell_1\ge2,\\[2pt]
    \bigl\lceil \ell_k/2\bigr\rceil+2, & \text{if } k\ge4,\ \ell_1=1
      \text{ and \eqref{eq:exclist} holds},\\[2pt]
    \bigl\lceil \ell_k/2\bigr\rceil+1, & \text{otherwise.}
  \end{cases}
\]
\end{theorem}

%% file: conclusion.tex
\section{Concluding remarks}
\label{sec:conclusion}

We proved that, for every labeling of a tree, all list-distance consistent vertices lie on a single path,
and used this structural restriction to determine the list-distance consistency number for complete $k$-ary trees
and for all spiders. The path theorem also gives a useful starting point for
computing this parameter on broader classes of trees. A systematic study of
the algorithmic complexity of determining $\ldc(T)$, including specialized
algorithms for caterpillars and related tree classes, is left for separate
work.

%% file: ai-declaration.tex
\clearpage
\section*{Declaration on the use of generative AI}

During the preparation of this manuscript, the authors used generative AI
tools, including Claude (Anthropic) and ChatGPT (OpenAI), for language editing,
literature searches, and assistance in writing and running computer programs
used to check the statements of several lemmas and theorems on small trees and
spiders. All mathematical definitions, statements, proofs, and conclusions
were formulated, reviewed, and verified by the authors, who take full
responsibility for the content of the manuscript.